\documentclass[reqno]{amsart}
\newcommand \datum {December 14, 2021}

\usepackage{amssymb,latexsym}
\usepackage{amsmath}
\usepackage{graphicx}
\usepackage{color}
\usepackage[dvipsnames]{xcolor} 
\usepackage{multibib} 
\usepackage{enumerate}
\numberwithin{equation}{section}
\theoremstyle{plain}
 \newtheorem{theorem}{Theorem}[section]

 \newtheorem{corollary}[theorem]{Corollary}
\theoremstyle{definition}
 \newtheorem{definition}[theorem]{Definition}
 \newtheorem{convention}[theorem]{Convention}
 \newtheorem{example}[theorem]{Example}

\theoremstyle{remark}
 \newtheorem{case}{Case} 
 \newtheorem{rcase}{Case} 
  
\newenvironment{enumeratei}{\begin{enumerate}[\quad\upshape (i)]} {\end{enumerate}}
\newcommand\nothing [1] {}
\newcommand\bdia {$\mathcal C_1$}
\newcommand \ucov [1] {#1^{+}}
\newcommand \lcov [1] {#1^{-}}
\newcommand \Sn[1] {S_7^{(#1)}}
\newcommand \gleq {\leq_{\textup{geom}}}
\newcommand \gless {<_{\textup{geom}}}
\newcommand \ray[2]{\vec\ell_{#1#2}}
\newcommand \nwray[1]{\vec\ell(\mathord\nwarrow#1)}
\newcommand \neray[1]{\vec\ell(#1\mathord\nearrow)}
\newcommand \swray[1]{\vec\ell(\mathord\swarrow#1)}
\newcommand \seray[1]{\vec\ell(#1\mathord\searrow)}
\newcommand \abp {\textup{AP}}
\newcommand \upcone[1] {\textup{Cone}^{\textup{up}}(#1)}
\newcommand \dncone[1] {\textup{Cone}_{\textup{dn}}(#1)}
\newcommand \abul {A^\bullet}
\newcommand \xstar {X^\star}
\newcommand \spc  {\hskip 7 pt}
\newcommand \rslope[2] {#1\nearrow_\textup{\kern -4pt sl}#2}
\newcommand \lslope[2] {#2\mathrel{_\textup{sl}\kern-4pt\mathord\nwarrow}#1}

\newcommand \tbf[1]  {\textbf{#1}} 
\newcommand \Jir [1] {\textup J(#1)} 
\newcommand \Mir [1] {\textup M(#1)} 
\newcommand \Nplu {\mathbb N^+}
\newcommand \set [1]{\{#1\}}
\newcommand \lbound[1] {B_{\textup{left}}(#1)}
\newcommand \rbound[1] {B_{\textup{right}}(#1)}

\newcommand \restrict [2] {#1\rceil_{\kern -1pt #2}}
\newcommand \ideal [1] {\mathord{\downarrow}#1}
\newcommand \filter [1] {\mathord{\uparrow}#1}

\begin{document}
\title[Meets and retracts in slim semimodular lattices]
{A property of meets in slim semimodular lattices and its application to retracts}

\author[G.\ Cz\'edli]{G\'abor Cz\'edli}
\email{czedli@math.u-szeged.hu \qquad Address$:$~University of Szeged, Hungary}
\urladdr{http://www.math.u-szeged.hu/~czedli/}
\nothing{\address{ Bolyai Institute, University of Szeged, Hungary}}

\begin{abstract} Slim semimodular lattices were introduced by G.\ Gr\"atzer and E.\ Knapp in 2007, and they have intensively been studied since then. It is often reasonable to give these lattices by their  $\mathcal C_1$-diagrams defined by the author in 2017. We prove that if $x$ and $y$ are incomparable elements in such a lattice $L$, then the interval $[x\wedge y, x]$ is a chain and this chain is of a normal slope in every $\mathcal C_1$-diagram of $L$. Except possibly for $x$, the elements of this chain are meet-reducible.
If $A $ and $X$ are subsets of a lattice $K$, then a sublattice $S$ of a lattice $L$ has the absorption property $(K,A ,X)$ if for every embedding $g\colon K\to L$ such that $g(A)\subseteq S$, we have that $g(X)\subseteq S$. If there is an idempotent endomorphism $f: L\to L$ such that $S=f(L)$, then the sublattice $S$ is a retract of $L$. 
Applying the above-mentioned property of meets, we present two absorption properties that  the retracts of every slim semimodular lattice $L$ have. 
\end{abstract}

\thanks{This research was supported by the National Research, Development and Innovation Fund of Hungary under funding scheme K 134851.}



\date{\datum.\hfill{{Hint: check the author's website for preprints and possible updates}}}

\maketitle

\section{Introduction}\label{sect:intro} 
Slim semimodular lattices were introduced by G.\ Gr\"atzer and E.\ Knapp in 2007. These lattices can be defined in two equivalent ways. According to the original definition, they are finite, planar, semimodular lattices that contain no $M_3$-sublattices; $M_3$ denotes the five-element modular lattice with three atoms. However, we prefer to go after Cz\'edli and Schmidt~\cite{czgschtJH}, where a finite (note necessarily semimodular) lattice $L$ is \emph{slim} if the
set $\Jir L$ of its (nonzero) join-irreducible elements is the union of two chains. We know from \cite[Lemma 2.2]{czgschtJH} that slim lattices are planar. In our setting, 
\emph{slim semimodular} lattices are the slim and semimodular lattices.  

At the time of writing, four dozen publications have been devoted to slim semimodular lattices; see 
the (extended) arXiv version of Cz\'edli~\cite{czgrefFgl} for their list\footnote{or see \ 
\texttt{http://www.math.u-szeged.hu/\textasciitilde{}czedli/m/listak/publ-psml.pdf}}.
Many of these publications are freely available and explain what motivates the study of these lattices and what connections these lattices have with other parts of mathematics; we only refer to 
Cz\'edli and Kurusa \cite{czgkurusa} and the book chapter Cz\'edli and Gr\"atzer~\cite{czgggltsta} for surveys. 

For an algebra $A$, the idempotent endomorphisms $f\colon A\to A$ are called the \emph{retractions} of $A$. That is, a retraction is a homomorphism from $A$ to itself such that $f(f(x))=f(x)$ for all $x\in A$. The \emph{retracts} of $A$ are the images $f(A)=\set{f(x):x\in A}$ of $A$ under the retractions $f$ of $A$. Retracts are particular subalgebras.  
The concept of retracts is similarly defined for  other categories of structures. Retracts are particularly important, for example, for posets (partially ordered sets); we only mention Rival~\cite{rival} and Z\'adori~\cite{zadori}. Apart from some obvious cases like vector spaces over a field (where every subspace is a retract) and monounary algebras, whose retracts have nice properties by  Jakub\'{\i}kov\'a--Studenovsk\'a and P\'ocs~\cite{danicap}, we do not know much about  retracts in general. 

For lattices, retractions and retracts have already been investigated in some papers including Boyu~\cite{liboyu}, Cz\'edli~\cite{czg-slimpatchabsretr}, \cite{czgcomes}, and  Cz\'edli and Molkhasi~\cite{czgmolkhasi}, but we still know little about them. In particular, there are only few lattices the retracts of which are well understood.

\subsection*{Goal and outline}
In Section~\ref{sect:sps}, we formulate and prove Theorem~\ref{thmmain} on  meets in slim semimodular lattices and slim lattices. A part of this theorem is purely algebraic, but it also has a visual part based \bdia-diagrams introduced in Cz\'edli~\cite{czg-diagr}. The paper is intended to be self-contained for lattice theorists; the necessary details about these diagrams will be given in due course.

In Section~\ref{sect:propsps}, we apply  Theorem~\ref{thmmain} to prove that the retracts of slim semimodular lattices have two particular absorption properties.

\section{A property of meets in slim semimodular lattices}\label{sect:sps}

Before formulating the main result of the paper, we recall the concept of \bdia-diagrams. These diagrams together with even more specific diagrams were introduced in Cz\'edli~\cite{czg-diagr}, and they have already proved to be efficient tools to study slim semimodular lattices; see, for example, Cz\'edli~\cite{czglamps}.

\begin{definition} (A) We always assume that a classical coordinate system of the plane is fixed. Lines or line segments parallel to $\set{(x,x): x\in \mathbb R}$ and those parallel to  $\set{(x,-x): x\in \mathbb R}$ are of \emph{normal slopes} $1$ and $-1$, respectively. These two slopes are said to be \emph{orthogonal}. 
The angle they make with $\set{(x,0):  0\leq x\in\mathbb R}$ is $\pi/4$ ($45^\circ$) and $3\pi/4$ ($135^\circ$), respectively. Lines or edges making an angle $\alpha$ with $\set{(x,0):  0\leq x\in\mathbb R}$ such that $\pi/4 < \alpha < 3\pi/4$ are called \emph{precipitous}. For example, vertical lines are such.

(B) Let $L$ be a slim semimodular lattice; we always assume that a planar diagram of $L$ is fixed. The left boundary chain and the right boundary chain of $L$ are denoted by
$\lbound L$ and $\rbound L$, respectively. (Their dependence on the diagram will not cause any trouble since the diagram is fixed.)  The set of non-unit meet-irreducible elements of  $L$ is denoted by $\Mir L$. 

(C) The planar diagram of $L$ is a \emph{\bdia-diagram} if every edge $[a,b]$ such that $a\in \Mir L\setminus(\lbound L\cup \rbound L)$ is precipitous and all other edges are of normal slopes. 

(D) If an interval $[u,v]$ of $L$ is a chain such that the edges of this chain are of the same normal slope, then we say that \emph{the interval $[u,v]$ is of normal slope}; otherwise $[u,v]$ has no slope.
\end{definition}

The definition of \bdia-diagrams above is easier to read than that in \cite{czg-diagr}, where several other diagrams are also defined. For another variant of the definition, the reader can (but need not) see Gr\"atzer~\cite{ggC1diagr}.

We know from Cz\'edli~\cite[Theorem 5.5(ii)]{czg-diagr} that each slim semimodular lattice has a \bdia-diagram. This allows us to stick to the following convention.

\begin{convention}\label{conv:fxd}
From now on, we assume that every slim semimodular lattice occurring in the paper has a \emph{fixed}  \bdia-diagram.
\end{convention}

Based on Convention~\ref{conv:fxd}, we are in the position to formulate the main result of the paper; let us emphasize that semimodularity is only assumed in its second part.

\begin{theorem}\label{thmmain}
Let $a$ and $b$ be incomparable elements of a slim lattice $L$, and let $c:=a\wedge b$. Then the following hold.
\begin{enumeratei}
\item\label{thmmaina} The intervals $[c,a]$ and $[c,b]$ are chains. 
\item\label{thmmainb} If, in addition,  $L$ is a slim \emph{semimodular} lattice, then  the intervals $[c,a]$ and $[c,b]$ are of normal slopes, their slopes are orthogonal, and 
every element of $([c,a]\setminus \set a) \cup ([c,b]\setminus\set b)$ is meet-reducible.
\end{enumeratei}
\end{theorem}

Without semimodularity, we cannot claim that the elements of  $([c,a]\setminus \set a) \cup ([c,b]\setminus\set b)$ in a slim lattice are meet-reducible. To see this, take the five-element nonmodular lattice $N_5$, which is slim, and let $a$ and $b$ be the coatoms of $N_5$.

\begin{proof}[Proof of Theorem \ref{thmmain}] 
Since $L$ is slim, $\Jir L$ is of the form $\Jir L=U\cup V$ where $U$ and $V$ are chains. We can assume that $L$ itself is not a chain. Replacing $V$ by $V\setminus U$ is necessary, we can assume that $U \cap V=\emptyset$. Note at this point that neither $U$ nor $V$ is empty since otherwise $L$ would be a chain. 
Since $U$ and $V$ are chains, we can write that 
$U=\set{u(1),u(2),\dots, u(m)}$ and $V=\set{v(1),\dots, v(n)}$ where $u(1)<u(2)<\dots<u(m)$ and $v(1)<v(2)<\dots<v(n)$.
For $x\in L$, let $i_x$ denote the largest (meaningful) subscript $s$ such that $u(s)\leq x$; if there is no such $s$ then $x_i:=0$. Similarly, $j_x$ stands for largest subscript $t$ such that $v(t)\leq x$; again, $x_j:=0$ if there is no such $t$.
Since each element is a join of join-irreducible elements, 
\begin{equation} x=u(i_x)\vee v(j_x),\,\,\,\text{ and }\,\,\,   x\leq y \iff (i_x\leq i_y\text{ and } j_x\leq j_y).\label{eq:smwHrqn}
\end{equation}
Observe that $u(s)\leq x\wedge y$ if and only if $u(s)\leq x$ and $u(s)\leq y$. The same holds for $v(t)$. Hence, it follows from \eqref{eq:smwHrqn}
 that, for any $x,y\in L$,
\begin{equation}
i_{x\wedge y}=\min\set{i_x,i_y}\quad\text{and}\quad j_{x\wedge y}=\min\set{j_x,j_y}.
\label{eq:mtsznK}
\end{equation} 
Now assume that $a,b\in L$ such that $a$ and $b$ are incomparable, in notation, $a\parallel b$.
Using \eqref{eq:smwHrqn}, we obtain that either $i_a<i_b$ and $j_a>j_b$, or $i_a>i_b$ and $j_a<j_b$. By symmetry, we can assume that $i_a<i_b$ and $j_a>j_b$. Let $c:=a\wedge b$. 
It follows from \eqref{eq:mtsznK} that $i_c=i_a$ and $j_c=j_b$. Now it is clear by \eqref{eq:smwHrqn} that for any $x\in 
[c,a]$, $i_x=i_a$. Therefore, $[c,a]$  is a chain by the first half of \eqref{eq:smwHrqn}. So is $[c,b]$ by symmetry.   This proves part \eqref{thmmaina}.

\begin{figure}[htb]
\centerline
{\includegraphics[width=\textwidth]{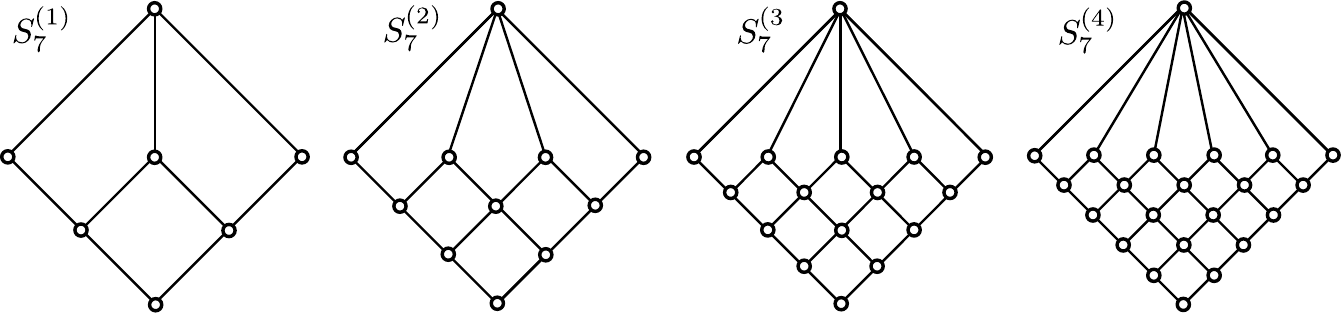}}      
\caption{$\Sn n$ for $n\in\set{1,2,3,4}$}\label{figsn}
\end{figure}

\begin{figure}[htb]
\centerline
{\includegraphics[width=\textwidth]{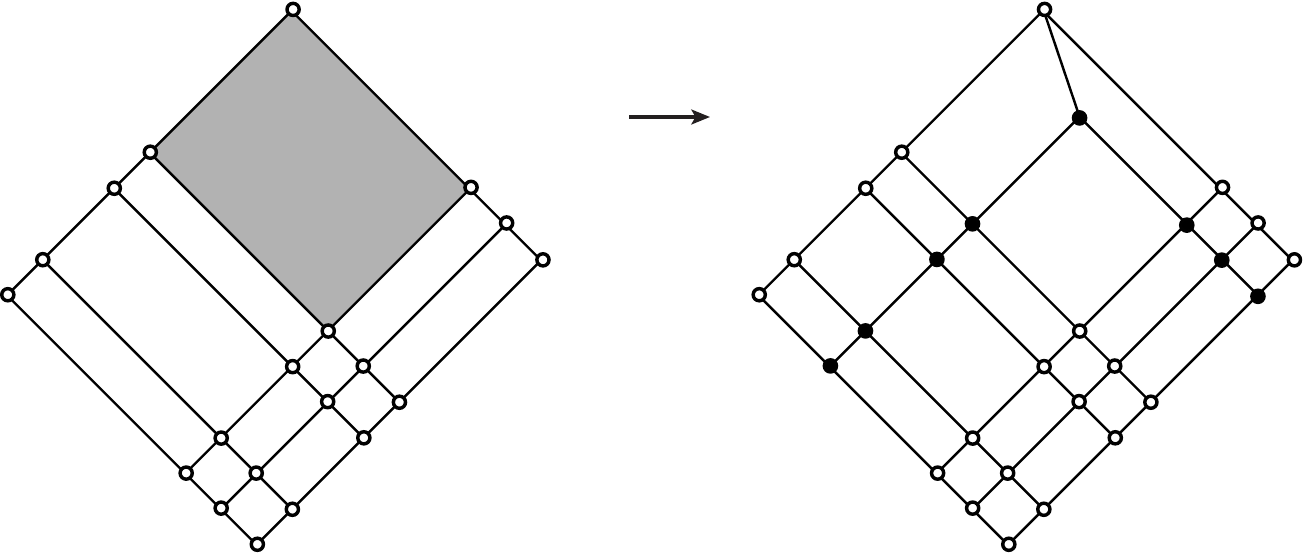}}      
\caption{A 1-fold multifork extension; in other words, a fork extension}\label{figmfa}
\end{figure}

\begin{figure}[htb]
\centerline
{\includegraphics[width=\textwidth]{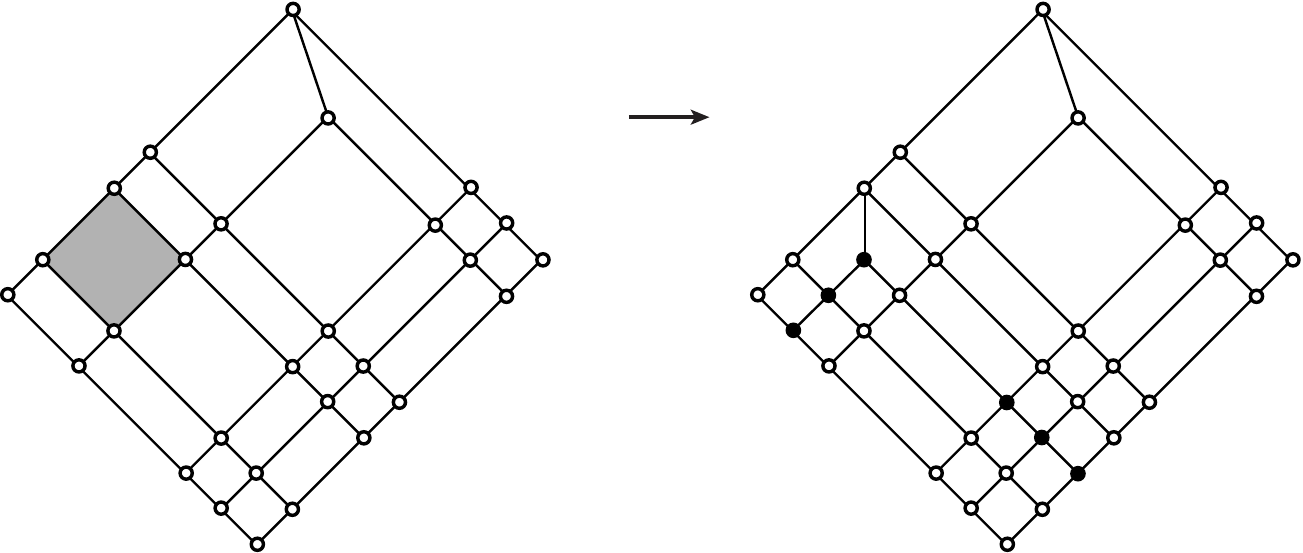}}      
\caption{Another 1-fold multifork extension}\label{figmfb}
\end{figure}

\begin{figure}[htb]
\centerline
{\includegraphics[width=\textwidth]{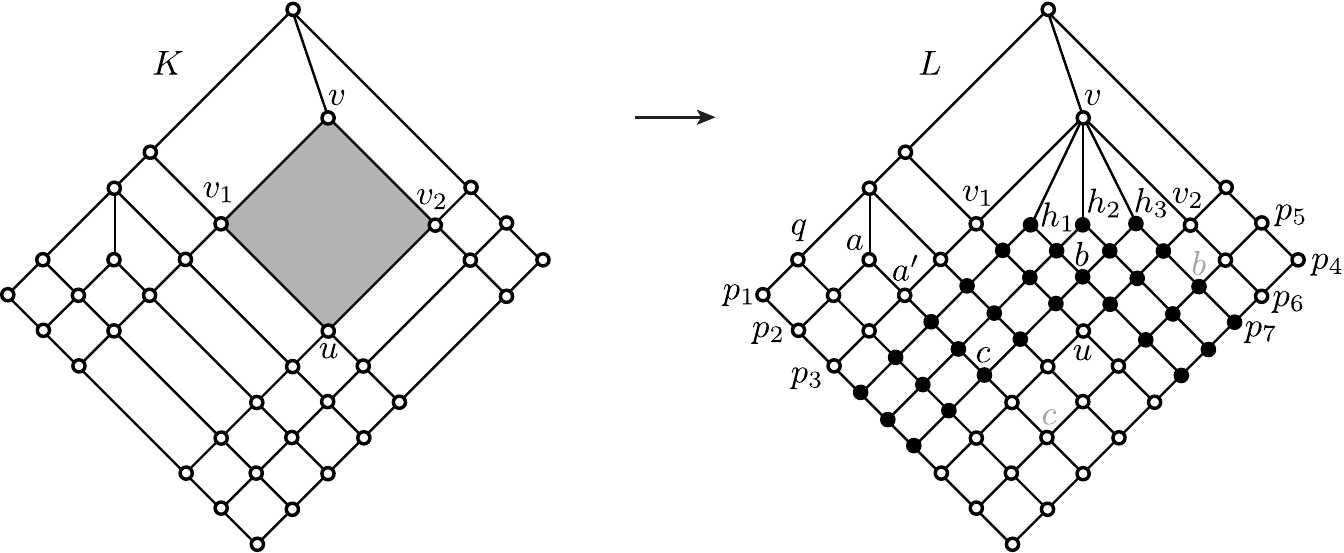}}      
\caption{A 3-fold multifork extension}\label{figmfc}
\end{figure}

To prove part \eqref{thmmainb}, we need to recall the structure theorem of slim semimodular lattices; it is  Cz\'edli~\cite[Theorem 3.7]{czg-patchext} (see also Cz\'edli \cite[Lemma 5.7]{czg-diagr} for diagrams) combined with Cz\'edli and Schmidt~\cite{czgschtvisual}. But first we need some concepts and notations. In the rest of the proof, let $L$ be a slim semimodular lattice. For $x\in \Jir L$ and $y\in \Mir L$, the unique lower cover of $x$ and the unique upper cover of $y$ are denoted by  $\lcov x$ and $\ucov x$, respectively. The elements of $\Jir L\cap\Mir L$ are called \emph{doubly irreducible}. By a \emph{corner} of a slim semimodular lattice $L$ we mean a doubly irreducible element  $u\in \lbound L\cup\rbound L$ such that $\ucov u$ has exactly two lower covers. 
For example, $L$ in Figure~\ref{figmfc} has exactly two corners, $p_1$ and $p_4$, but $q$ is neither a corner of $L$, nor  a corner of the sublattice $L\setminus\set{p_1,p_2,p_3}$. 
By a \emph{grid} we mean the direct product of two finite nonsingleton chains.
The diagram of $L$ is divided into so-called \emph{$4$-cells} by edges; these $4$-cells are four-element lattices and they are also intervals of length 2. The top element of a $4$-cell $X$ is denoted by $1_X$. We say that a $4$-cell $X$ is a \emph{distributive $4$-cell} if the principal ideal $\ideal 1_X:=\set{u\in L: u\leq 1_X}$ is a distributive lattice.
The lattices $\Sn n$ for $n=1,2,3,4$ are given in Figure \ref{figsn}; $\Sn n$ for $n\geq 5$ is analogously defined in Cz\'edli~\cite{czg-patchext}.
Figures~\ref{figmfa}--\ref{figmfc} show what a \emph{multifork extension}, introduced in \cite{czg-patchext}, is. Namely, 
\begin{itemize}
\item[---] first we choose an $n\in\Nplu:=\set{1,2,3,\dots}$ ($n$ is 1, 1, and 3 in Figures \ref{figmfa}, \ref{figmfb}, and \ref{figmfc}, respectively); in other words, two 1-fold multifork extensions are followed by a 3-fold one).
\item[---] Second, we pick a distributive 4-cell (the grey one on the left), which we change to a copy of $\Sn n$ (see on the right); the new elements that enter are called \emph{central} (they are $h_1$, $h_2$, $h_3$, $b$, and the two covers of $b$ in Figure~\ref{figmfc}), 
\item[---] Finally, proceeding to the lower left and the lower right directions, we add further elements, the so-called \emph{left new elements} and \emph{right new elements}, to keep semimodularity. 
\end{itemize}
In Figures \ref{figmfa}--\ref{figmfc}, the  new elements are the black-filled ones. As Figures~\ref{figmfa}--\ref{figmfc} show, $L$ in Figure~\ref{figmfc} is obtained from a grid by three consecutive multifork extensions. Now the structure theorem of slim semimodular lattices asserts the following.
\begin{equation}\left.
\parbox{9.5cm}{Each slim semimodular lattice and its \bdia-diagram can be obtained such that we take a grid, perform multifork extensions at distributive $4$-cells in a finite number of steps, and then we omit some corners one by one. Furthermore, any lattice obtained in this was is a slim semimodular lattice.
}\,\,\right\}
\label{pbx:strthm}
\end{equation}
It is important that multifork extensions in \eqref{pbx:strthm} are only allowed \emph{before} omitting corners. 
It is allowed to perform no multifork extension or  omit no corner. 

For distinct elements $b$ and $c$ of  a slim semimodular lattice $L$, keeping Convention~\ref{conv:fxd} in mind, let $\ray bc$ denote 
the ray (also called half-line) through $c$ with initial point $b$; $\ray bc$ is a geometric object. We write $b\gless c$ to denote that $b\neq c$ and the angle between the positive half of the $x$-axis and  $\ray bc$ is in the interval $[\pi/4, 3\pi/4]$ (that is, between $45^\circ$ and $135^\circ$). Naturally, $b\gleq c$ means that $b\gless c$ or $b=c$. For $b\in L$, we define the \emph{upper cone} and the \emph{lower cone} of $b$ as the following subsets of the plane:
\[\upcone b:=\set{c\in\mathbb R^2: b\gleq c}\text{ and } 
\dncone b:=\set{c\in\mathbb R^2: c\gleq b}. 
\]
We know from Cz\'edli~\cite[Corollary 6.1]{czg-diagr} that, for any $x,y\in L$, 
\begin{equation}
x\leq y \iff x\gleq y \iff x\in\dncone y\iff y\in\upcone x.
\label{eq:kszBfhtlR}
\end{equation}


Next, we recall a concept from Kelly and Rival~\cite{kellyrival}. For $u,v\in L$ such that $u\parallel v$, we say that $u$ is \emph{to the left of} $v$ if $u$ is on the left of some (equivalently, every) maximal chain containing $v$. 
We know from Kelly and Rival~\cite[Lemma 1.2 and Proposition 1.7]{kellyrival}
that, for all elements $u,v$ in any planar lattice,
\begin{align}
\parbox{10cm}{if $u\parallel v$, 
then $u$ is to the left of $v$ or $v$ is to the left of $u$, and}
\label{eqtxt:lfmrHv}\\
\parbox{10cm}{if $C$ is a maximal chain of $L$, $u$ is on the left while $v$  is on the right of $C$, and $u<v$, then $u\leq w\leq v$ holds for some $w\in C$.}
\end{align}
Now we are in the position to prove part \eqref{thmmainb} of the theorem. In virtue of  \eqref{pbx:strthm}, it suffices to show that
\begin{itemize}
\item[(a)] \eqref{thmmainb} holds for every grid,
\item[(b)] if \eqref{thmmainb} holds for a slim semimodular lattice $K$ and $L$ is obtained from $K$ by a multifork extension at a distributive $4$-cell, then \eqref{thmmainb} also holds for $L$, and
\item[(c)] if \eqref{thmmainb} holds for a slim semimodular lattice $K$ and $L$ is obtained from $K$ by omitting a corner, then  \eqref{thmmainb} holds for $L$, too.
\end{itemize}

The validity of (a) is trivial. 

Assume that $K$ and $L$ are as described in (b) and  \eqref{thmmainb} holds for $K$; see Figure~\ref{figmfc} where $L$ is obtained from $K$ by a 3-fold multifork extension. The new elements, that is, the elements of $L\setminus K$ are the black-filled ones. From left to right, the new meet-irreducible elements are labeled by $h_1$, \dots, $h_n=h_3$. Since the multifork extension is performed at a distributive 4-cell with top element denoted by $v$,  the principal ideals $\ideal_K v$, $\ideal_L h_1$, \dots, $\ideal_L h_n$ are grids.  For convenience, 
\begin{equation}
\parbox{9cm}{if $x\parallel y$, $z=x\wedge y$, and part \eqref{thmmainb} holds if we let $(a,b,c):=(x,y,z)$, then we say that $z=x\wedge y$ is a \emph{\bdia-regular meet}.}
\label{pbx:rmT}
\end{equation}
Let $a$ and $b$ be incomparable elements of $L$, and let $c:=a\wedge b$. We have to show that $c=a\wedge b$ is a \bdia-regular meet. There are several cases.

\begin{case}\label{case:egy} $a,b\in K$. Then $c\in K$, the chain $C:=[c,a]_K$ of $K$ is of a normal slope, and its elements except possibly $a$ are meet-reducible. The edges (as line segments) of $C$ can be divided into shorter edges by some new elements $x_1,\dots,x_t$  in $L$, but $C$ as a line segment and so its  slope remain the same. It is clear by construction and Figures~\ref{figmfa}--\ref{figmfc} that none of $h_1,\dots, h_n$ can divide an old edge. Hence $\set{x_1,\dots,x_t}\cap\set{h_1,\dots,h_n}=\emptyset$. Since all  the new elements but $h_1,\dots,h_n$ are meet-reducible in $L$, so are $x_1,\dots,x_n$. Since a meet-reducible element of $K$ is also meet-reducible in $L$ and $a$ and $b$ play a symmetrical role , we conclude that $c=a\wedge b$ is a \bdia-regular meet.
\end{case}

\begin{case}\label{case:ketto} $a,b\in L\setminus K$: since
$\ideal_K v$ is distributive, it is a grid. Hence, 
 $\ideal_L h_1$, \dots, $\ideal_L h_n$ are also grids; see Figure~\ref{figmfc}. Let $T:=\ideal_L h_1\cup \dots \cup\ideal_L h_n$.
Then $T$ is a meet-subsemilattice and a subdiagram of a larger grid $G$. This $G$  is not a subset of $L$ but no problem: $c=a\wedge b$ is a \bdia-regular meet in $G$ by (a), and $T$ is a common meet-subsemilattice of $L$ and $G$. Since $h_1$, \dots, $h_n$ are maximal elements in $T$, they are not in $([c,a]\setminus\set a)\cup ([c,b]\setminus\set b)$, and it follows easily that $c=a\wedge b$ is a \bdia-regular meet in $L$. 
\end{case}

\begin{case}\label{case:harom} $a\in K$ and $b\in L\setminus K$. 
Two possible positions of $b$ and $c:=a\wedge b$ are given in Figure~\ref{figmfc}, one with black letters and another one with light-grey letters. 
Since $b$ is a new element, $b<v$. 
Denote by $v_1$ and $v_2$ the lower covers of $v$ in $K$; see Figure~\ref{figmfc}.  Using \eqref{eq:kszBfhtlR} and reflecting the diagram across a vertical axis if necessary, we can assume that $a$ is to the left of $b$. (Then the situation complies with the figure.)  Since $b < v$ and $a\parallel b$, we have that $v\not\leq a$. Hence,  there are two subcases.

In the first subcase, we assume that $a<v$. Observe that
\begin{equation}
\parbox{8cm}{$\ideal_L v\setminus\set v=\ideal_L v_1\cup \ideal_L v_2\cup \ideal_L h_1\cup \dots\cup\ideal_L h_n$, and every principal ideal on the right of the equality sign is a chain with a normal slope or a grid.}
\end{equation}
As in Case~\ref{case:ketto}, we can extend $T:=\ideal_L v\setminus \set v$ to a grid $G$ such that $T$ is a common meet-subsemilattice of $G$ and $L$. Then we can conclude
the \bdia-regularity of $c=a\wedge b$ basically in the same way as in Case~\ref{case:ketto}.

In the second subcase, we assume that $a\parallel v$. Since $b < v$ gives the existence of a maximal chain containing both $b$ and $v$, we obtain that $a$ is to the left of $v$. Let $a':=a\wedge v$, and observe that  $a'\wedge b=(a\wedge v)\wedge b= a\wedge (v\wedge b)=a\wedge b=c$.
This fact, $a'<v$, and the previous subcase yield that $c=a'\wedge b$ is a \bdia-regular meet. In particular, the interval $[c,b]$ has the required properties. So does $[c,a']$, but it is only a part of $[c,a]$.  We know from the already proven part \eqref{thmmaina} of the theorem that $[c,a]$ is an interval. Thus, $a'$ is comparable with all elements of $[c,a]$, and it follows that $[c,a]$ is the union of $[c,a']$ and $[a',a]$. By its definition, $a'$ is meet-reducible. Since part \eqref{thmmainb} holds for $K$ and $a'=a\wedge v$, the intervals $[a',a]$ and $[a',v]$ are of orthogonal normal slopes. We have seen that both $[a',a]$  and $[c,a']$ are of normal slopes, but we have to verify that they are of the same normal slope. Since the slope of $[a',a]$ is orthogonal to that of $[a',v]$, it suffices to show that so is the slope of $[c,a']$. Suppose not. Then $[c,a']$ and $[a',v]$ is of the same normal slope. Hence, $c$ lies on the geometrical line of a normal slope through $a'$ and $v$.
(So, as opposed to what Figure~\ref{figmfc} shows, $c$ is on the line segment from $p_3$ to $v$.) Therefore, by the definition of multifork extensions, the principal filter $\filter_L c$ cannot contain any new element. This contradicts $b\in \filter_L c$ and concludes Case~\ref{case:harom}.  
\end{case}

Cases \ref{case:egy}--\ref{case:harom} yield the validity of 
(b).

To prove (c), assume that $w$ is a corner of $K$ and $L=K\setminus \set w$. Let $a,b\in L$ such that $a\parallel b$. Then $c:=a\wedge b$ is a \bdia-regular meet in $K$. Note that $w\notin\set{a,b,c}\subseteq L$. 
We know that $[c,a]_K$ and $[c,b]_K$ are chains of normal slopes. 
To see that $[c,a]_L=[c,a]_K$, it suffices to show that $w\notin[c,a]_K$.  But this is obvious since otherwise $[c,a]_K$ would make an ``orthogonal turn" at $w$ and it could not be of a normal slope in $K$. Similarly, $[c,b]_L=[c,b]_K$. 

The only meet-reducible element of $K$ that turns into meet-irreducible in $L$ is $x_0:=\lcov w$. Hence, we need to show that none of  $[c,a]_K\setminus\set a$ and $[c,b]_K\setminus\set b$ contains $x_0$. By symmetry, it suffices to deal with $[c,a]_K\setminus\set a$, and we can assume that  $w\in\lbound K$. For the sake of contradiction, suppose that $c\leq x_0<a$. For convenience, we introduce the following notation.
For an edge $[p,q]$, let $\rslope p q$ and $\lslope p q$ mean the $[p,q]$ is of slope 1  ($45^\circ$) or slope $-1$ ($135^\circ$), respectively.
 Let $x_1$ denote the cover of $x_0$ that is distinct of $w$. Since we are in a \bdia-diagram, 
both $[x_0,w]$ and $[x_0,x_1]$ are of normal slopes and these slopes are different. Using that $w\in\lbound L$, the only possibility is that $\lslope{x_0}w$ and $\rslope{x_0}{x_1}$.
Every element has at most two covers; this was proved by Gr\"atzer and Knapp~\cite[Lemma 8]{gratzerknapp1} (and also follows from the definition of  \bdia-diagrams). Therefore, since $w\notin [c,a]_K$ and $c\leq x_0<a$, we have that $x_1\in[c,a]_K$. 
From the facts that $[x_0,x_1]$ is an edge of the chain $[c,a]_K$, this chain is of a normal slope, and $\rslope{x_0}{x_1}$, we obtain that 
\begin{equation}
\text{for every edge $p\prec q$ of $[c,a]_K$, we have that $\rslope p q$.}
\label{eq:wmRhrs}
\end{equation} 
We claim that 
\begin{equation}
\text{if $p\prec q$ is an edge of $[c,a]_K$ and $q\in\lbound K$, then $p\in\lbound K\cap\Mir K$.}
\label{eq:msrTfjn}
\end{equation}
Suppose to the contrary that $p\notin \lbound K$. But $\lbound K$ is a maximal chain of $K$, whereby there is a unique $r\in \lbound K$  such that $r\prec q$. Since $r$ and $p$ are different lower covers of $q$, they are incomparable and one of them is to the left of the other by \eqref{eqtxt:lfmrHv}. Since $r$ is on the left boundary chain $\lbound K$, $p$ cannot be to left of $r$.
Hence, $r$ is to the left of $p$.  We know from \eqref{eq:wmRhrs} that $\rslope p q$. In other words, the angle that the edge $p\prec q$ and $\set{(x,0):  0\leq x\in\mathbb R}$ makes is $45^\circ$. Since $r$ is to the left of $q$, the edge $r\prec q$ makes a smaller angle. This contradicts the fact that we are in a \bdia-diagram. This proves that $p\in\lbound K$. 

To obtain a contradiction again, suppose that $q$ is not the only cover of $p$, and pick a cover $s$ of $p$ such that $s\neq q$. Since $q\parallel s$ and $q\in\lbound K$, we have that $q$ is to the left of $s$. Hence $[p,s]$ makes a smaller angle with $\set{(x,0):  0\leq x\in\mathbb R}$  than $[p,q]$. But this is a contradiction since, by \eqref{eq:wmRhrs},  $[p,q]$ makes the smallest possible angle, $45^\circ$. Therefore, $p\in\Mir K$, proving  \eqref{eq:msrTfjn}.  

Let $y$ denote the unique cover of $c$ that belongs to the chain $[c,a]_K$. Since $x_0\in\Mir L$ but $c\notin\Mir L$,  we obtain that $c<x_0\in[c,a]_K$. Hence $y\leq x_0$. 
Descending along the chain $[c,a]_K$ from $x_0\in\lbound K\cap\Mir K$ down to $y$, we obtain by \eqref{eq:wmRhrs} and \eqref{eq:msrTfjn} that $c\in\Mir K$. This is a contradiction since the meet-reducibility of $c$ in $L$ implies its meet-reducibility in $K$.
Therefore, (c) holds, and the proof of Theorem~\ref{thmmain} is complete.
\end{proof}

\section{Two properties of retracts of slim semimodular lattices}\label{sect:propsps}
We introduce a class of properties by defining which lattices satisfy them. As usual, for a set $Y$ and a map $h$, the set $\set{h(y):y\in Y}$ is denoted by $h(Y)$.

\begin{definition}\label{def:pRpC}
Assume that $K$ is a lattice, $\abul$ is a sublattice  of $K$, and $\xstar$ is a nonempty subset of $K$. (If we want to avoid redundancy and triviality, we also assume that $\abul\neq K$ and $\abul\cap\xstar=\emptyset$.) Let $L$ also be a lattice.

\begin{itemize}
\item[$\square$] A sublattice $S$ of $L$ satisfies
the \emph{absorption property} $\abp(K,\abul,\xstar)$ in $L$ if for every embedding $g\colon K\to L$, the inclusion $g(\abul)\subseteq S$ implies that $g(\xstar)\subseteq S$. 

\item[$\square$] If \emph{every} retract of $L$ satisfies $\abp(K,\abul,\xstar)$ in $L$, then we say that the \emph{retracts} of $L$ satisfy
the absorption property $\abp(K,\abul,\xstar)$.

\item[$\square$] In this paper, $\abp(K,\abul,\xstar)$ is always given by a single diagram: the diagram of $K$ in which the elements of $\abul$ are drawn by large black-filled circles while $\xstar$ is the set of star-shaped elements.
\end{itemize}
\end{definition}

For example, property $P(8,1)$ given in Figure~\ref{figprt} is the condition $\abp(K,\abul,\xstar)$ where $K=\set{a,b,c,d,x,y,z,t}$ is the lattice given in the figure, $\abul=\set{a,b,c,d}$, and $\xstar=\set{y}$. Using Theorem~\ref{thmmain}, we are going to prove the following two corollaries; $P(8,1)$ and $P(9,2)$ are given by Figure~\ref{figprt} and Definition~\ref{def:pRpC}.

\begin{figure}[htb]
\centerline
{\includegraphics[scale=1.0]{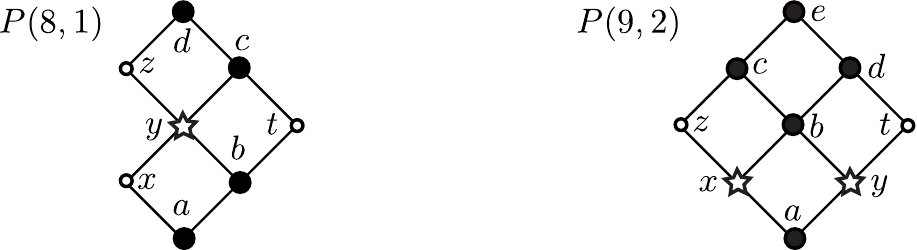}}     
\caption{Absorption properties $P(8,1)$ and $P(9,2)$, which hold for slim semimodular lattices
}\label{figprt}
\end{figure}

\begin{corollary}\label{cor:p81} The retracts of every slim semimodular lattice satisfy $P(8,1)$.
\end{corollary} 

\begin{corollary}\label{cor:p92} The retracts of every slim semimodular lattice satisfy $P(9,2)$.
\end{corollary}

\begin{proof}[Proof of Corollary \ref{cor:p81}] 
With $\abp(K,\abul,\xstar):=P(8,1)$, let $S$ be a retract of a slim semimodular lattice  $L$, and let $g\colon K\to L$ be an embedding as in Definition~\ref{def:pRpC}. We can assume that $g$ is the inclusion map. Then $K$ is a sublattice of $L$ and $\abul\subseteq S$; we need to show that $y\in S$. Convention \ref{conv:fxd} applies.
Pick a retraction map $f\to S$. For convenience, we will write $a',b', \dots, z'$ instead of $f(a)$, $f(b)$,  \dots, $f(z)$. Of course, we can drop the apostrophe at black-filled elements, for example, $a'=a$. The properties of $f$ like  $u\leq v \Rightarrow u'\leq v'$ will frequently be used without much explanation.
For $u\in L$, if there is a ray (as a geometric object) with a normal slope and with initial point $u$ such that $u$ is not the only element of $L$ adjacent to this ray, then we denote this ray by 
\begin{equation}
\nwray u,\quad  \neray u, \quad \swray u, \quad\text{or}\quad \seray u
\label{eq:raYdef}
\end{equation}
depending on its slope and ``up or down'' orientation. 
For example, in $L$ of Figure~\ref{figmfc}, $v_1$ is on $\nwray u$  (we will write  $v_1\in\nwray u$  to denote this adjacency), $u\in\seray{v_1}$,   $b\in\neray c$, and $c\in\swray b$.

\begin{figure}[htb]
\centerline
{\includegraphics[scale=1.0]{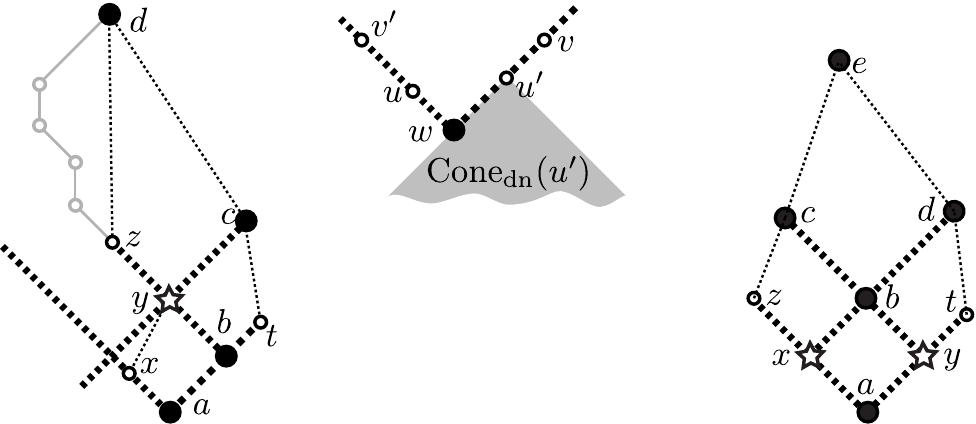}}     
\caption{Proving, from left to right, Corollary \ref{cor:p81},  \eqref{pbx:nSkmt}, and Corollary \ref{cor:p92}}
\label{figspr}
\end{figure}

By  Cz\'edli~\cite[Proposition 5.1]{czg-diagr},  we can assume that the left-right orientation of Figure~\ref{figspr} is correct since otherwise we can reflect the diagram across a vertical axis. Applying Theorem~\ref{thmmain} to $a=x\wedge t$, $b=z\wedge t$, and $y=z\wedge c$,  we obtain that 
\begin{equation}
x\in \nwray a,\spc b,t\in\neray a, \spc y,z\in\nwray b, \spc c\in\neray y;
\label{eq:smrRm}
\end{equation}
see the left side of Figure~\ref{figspr}. In this figure, the dotted lines are not for edges. Although, say, the segment of $\nwray b$ between $b$ and $z$ consists of edges by Theorem~\ref{thmmain}, it may happen that no edge lies on, say, the line segment connecting $c$ and $d$. However, the line segments indicate the ordering, the thick ones are of normal slopes, and each of the thin ones is precipitous or is of a normal slope.

We claim the following, which is the key idea of the proof. 
\begin{equation}
\parbox{8cm}{If $u\parallel v$ in $L$, $w:=u\wedge v\in S$, $u$ is to the left of $v$, and $u'\parallel v'$, then $u'\in\nwray w$, $v'\in\neray w$, and both $[w,u']$ and $[w,v']$ are chains of normal slopes.}
\label{pbx:nSkmt}
\end{equation}
To show this, observe that $w=w'=(u\wedge v)'=u'\wedge v'$. So, in the sense of \eqref{pbx:rmT}, both $w=u\wedge v$ and $w=u'\wedge v'$ are \bdia-regular meets by Theorem~\ref{thmmain}. In particular,  $[w,u']$ and $[w,v']$ are chains of normal slopes. Hence, it suffices to exclude that 
$u'\in\neray w$ and $v'\in\nwray w$. For the sake of contradiction, suppose that $u'\in\neray w$ and $v'\in\nwray w$, as in the middle of  Figure~\ref{figspr}. In the figure, the thick dotted rays are geometric half-lines and, say, the line segment from $w$ to $v'$ represents a  chain of a normal slope. Since $u$ is to the left of $v$ and $w=u\wedge v$ is a \bdia-regular meet, $u\in\nwray w$ and $v\in\neray w$. In the figure, the light-grey area indicates $\dncone {u'}$.
Since the dotted thick rays are of normal slopes, $\dncone {u}\cap \dncone {u'}= \dncone w$. Combining this equality with \eqref{eq:kszBfhtlR}, we obtain that $u\wedge u'=w$. Hence, 
$u'=u'\wedge u'=u'\wedge u''=(u\wedge u')'=w'=w$. On the other hand, $w\leq v$ yields that $w=w'\leq v'$. Thus, $u'=w\leq v'$ contradicts the assumption that $u'\parallel v'$, completing the proof of \eqref{pbx:nSkmt}.

Let us mention at this pont that, with self-explanatory changes, 
\begin{equation}
\parbox{6cm}{all what we have done in this proof so far will be needed in the next proof.}
\label{pbx:lLwhD}
\end{equation}

We need to show that $y'=y$ since this would mean that $y\in S$.
There are two cases discussed below; we are going to show that the first of these cases gives  the required $y'=y$ while the second one cannot occur since it leads to contradiction.

\begin{rcase}\label{rcase:zpt} We assume that $z'\parallel t'$.
Applying \eqref{pbx:nSkmt} with $(u,v,w):=(z,t,b)$, we obtain that
$[b,z']$ is a chain of a normal slope, $z'\in \nwray b$ and $t'\in\neray b$.  Since $z'\in \nwray b$ and $t\in\neray b\setminus\set b$, we have that $t\notin \dncone {z'}$. This fact and \eqref{eq:kszBfhtlR} give that $z'\not\geq t$, implying that $z'\not\geq c$. Thus, either $z'\parallel c$ or $z'<c$.  
However, if $z'<c$, then  $t'\leq c'=c$ leads to $d=d'=(z\vee t)'=z'\vee t'\leq c$, which is a contradiction. 
Therefore, $z'\parallel c$. Then $y'=(z\wedge c)'=z'\wedge c'=z'\wedge c$, so $y'=z'\wedge c$ is a \bdia-regular meet. By Theorem~\ref{thmmain},  
\begin{equation}
\text{either $y'\in \seray c$, or $y'\in \swray c$.}
\label{eq:twzFrm}
\end{equation}
Since $[b,z']$ is a chain of a normal slope and $b\leq y\leq z$ implies that $b=b'\leq y'\leq z'$, we obtain that $y'\in\nwray b$. 
The rays $\nwray b$ and $\seray c$ are parallel. The line segment between $y$ and $c$ is orthogonal to these parallel rays, and $y\in\nwray b$. Hence the geometric distance of $\nwray b$ and $\seray c$ is the distance of $y$ and $c$, which is positive since $y\neq c$. Thus, $\nwray b\cap \seray c=\emptyset$. This fact and $y'\in\nwray b$ excludes that $y'\in \seray c$, whereby \eqref{eq:twzFrm} implies that $y'\in \swray c$. So $y'$ belongs to both $\nwray b$ and $\swray c$. But these two orthogonal rays only have one point in common, which is $y$. Therefore $y'=y$, as required. This completes the analysis of Case~\ref{rcase:zpt}.
\end{rcase}

\begin{rcase}\label{rcase:hsnLt} We assume that $z'$ and $t'$ are comparable. Since their join and meet are $d=d'$ and $b=b'$, we have that $\set{z',t'}=\set{d,b}$. But $c=c'\geq t'$ since $c\geq t$, whereby $t'\neq d$ and we conclude that $z'=d$ and $t'=b$.  
There are two subcases 
to consider.

First, assume that $x'\parallel b$. Since $b'=b$, \eqref{pbx:nSkmt} 
gives that $x'\in\nwray a$. Let $C'$ be a maximal chain in the interval $[z,d]$, and take the chain $C:=[a,b]\cup [b,z]\cup C'$; this is indeed a chain since $[a,b]$ and $[b,c]$ are chains (of normal slopes) by Theorem~\ref{thmmain}. Now $b$ and every element of $[a,b]\setminus\set b$ is strictly on the right of the geometric ray $\nwray a$. (``Strictly on the right of'' means that ``on the right of but not belonging to''; the meaning of ``strictly on the left of'' is analogous.)  Since every edge of the subchain $C\cap[b,d]$ is precipitous or of a normal slope, every element of $C\cap[b,d]$ is strictly on the right of $\nwray a$, as the figure shows. Therefore,
\begin{equation}
\text{every element of $C\setminus \set a$ is strictly on the right of $\nwray a$.}
\label{eq:tSdzjbr}
\end{equation}
Let $L':=[a,d]$, and observe that $a\leq x\leq d$ gives that $a=a'\leq x'\leq d'=d$, that is, $x'\in L'$. Now $C$ is a maximal chain of the planar lattice $L'$, $b$ is on the left of $C$ (since $b\in C$), 
and $x'\in\nwray a$ combined with \eqref{eq:tSdzjbr} yield that $x'$ is also on the left of $C$.  It is visually clear and it has rigorously been proved in Kelly and Rival~\cite[Proposition 1.4]{kellyrival} that the elements on the left of a maximal chain of a planar lattice form a (convex) sublattice. Consequently, 
\begin{equation}
y'=(x\vee b)'=x'\vee b'=x'\vee b\text{ is on the left of $C$}.
\label{eq:syzmfFrm}
\end{equation}
Since $z\in\nwray y\setminus\set y$, $c\in\neray y\setminus y$, and every  edge of $C'$ is precipitous or of a normal slope, it follows that, again in $L'$, 
\begin{equation}
c \text{ is strictly on the right of } C.
\label{eq:srzBmrrncV}
\end{equation}
Finally, using that $z'=d$, we obtain that $y'=(z\wedge c)'=z'\wedge c' = d\wedge c=c$. Hence, $y'=c$, \eqref{eq:syzmfFrm}, and \eqref{eq:srzBmrrncV} simultaneously hold, which is a contradiction proving that the subcase $x'\parallel b$ cannot occur. 

As the second subcase, now we assume that  $x'$ and $b$ are comparable. Since $a=a'=(x\wedge b)'=x'\wedge b'=x'\wedge b$, the just-mentioned comparability yields that $x'=a$. Using this equality and $t'=b$, we obtain that 
$c=c'=(x\vee t)'=x'\vee t'=a\vee b=b$, which is a contradiction excluding the second subcase. Therefore, Case~\ref{rcase:hsnLt} cannot occur.
\end{rcase}

We have seen that Case~\ref{rcase:zpt} implies the required $y'=y$ while  Case~\ref{rcase:hsnLt} cannot occur. Thus, the proof of Corollary~\ref{cor:p81} is complete.
\end{proof}

\begin{proof}[Proof of Corollary \ref{cor:p92}] 
The relevant illustration is on the right of Figure~\ref{figspr}. 
After recalling \eqref{pbx:lLwhD}, first we show that $\set{x',y'}\neq\set{a,b}$. Suppose the contrary. Then, since $x$ and $y$ play a symmetrical role, we can assume that $x'=a$ and $y'=b$. Since $t'\geq y'=b$, \eqref{eq:kszBfhtlR} gives that $t'\in \upcone b$. But $\upcone b$ is in the (topological) interior of $\upcone a$, and both $\upcone a$ and $\upcone b$ are formed by rays of normal slopes. Hence $t'\in\upcone b$ implies that 
$a\notin \swray{t'}$ and $a\notin \seray{t'}$. Thus, Theorem~\ref{thmmain} excludes that $a=z'\wedge t'$ such that $z'\parallel t'$. 
The rest of the proof uses the same ideas as the proof of Corollary~\ref{cor:p81}; this version of the paper leaves the details to the reader.
\end{proof}

\begin{figure}[htb]
\centerline
{\includegraphics[scale=1.0]{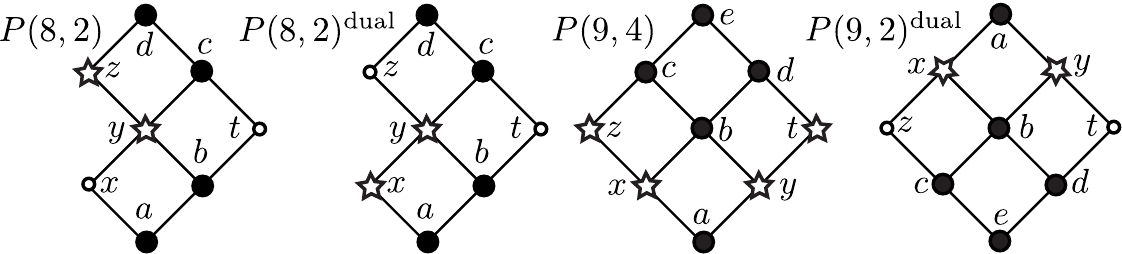}}     
\caption{Some absorption properties that, in general, do not hold for the retracts in a slim semimodular lattice
}\label{fignop}
\end{figure}

An absorption property $\abp(K,\abul,\xstar)$ becomes stronger if $\xstar$ is replaced by a larger subset of $K$. 
Note that different choices of $\xstar$ may give equivalent properties.
For example, it follows by symmetry that in case of $P(9,4)$ of Figure~\ref{fignop}, we would obtain an equivalent property if we changed $\set{x,y,z,t}$ to  $\set{x,y,z}$ or $\set{x,y,t}$. 
Therefore, the following example implies that  none of Corollaries \ref{cor:p81} and \ref{cor:p92} can be strengthened by taking a larger $\xstar$.

\begin{example}\label{ex:mpSlw}
None of the absorption properties $P(8,2)$, $P(8,2)^{\textup{dual}}$, $P(9,4)$, and $P(9,4)^{\textup{dual}}$
given in Figure~\ref{fignop} holds for for the retracts of all slim semimodular lattices.
\end{example}

\begin{figure}[htb]
\centerline
{\includegraphics[scale=1.0]{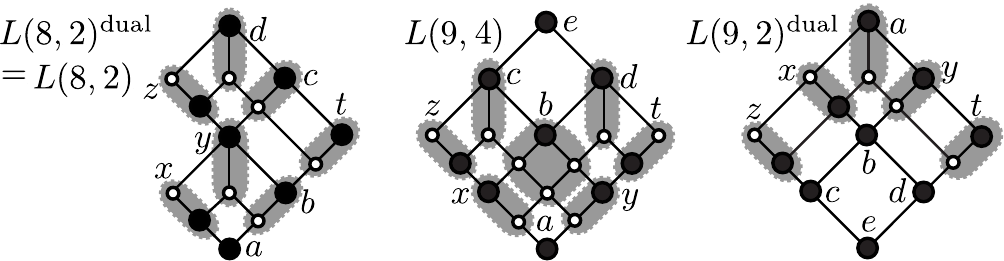}}     
\caption{To the proof of Example \ref{ex:mpSlw}
}\label{figprn}
\end{figure}

\begin{proof} Consider the lattices $L(8,2)=L(8,2)^{\textup{dual}}$, $L(9,4)$, and $L(9,4)^{\textup{dual}}$ in Figure~\ref{figprn}. They are slim semimodular lattices by (the last sentence of) \eqref{pbx:strthm}. For each of these lattices, define a map by the rule that $u\mapsto u$ if $u$ is a black-filled element, and $u\mapsto$ the unique black-filled element in the same grey ``oval cog-wheel'' otherwise. It is straightforward to check that this map is a retraction. Hence, we get an example showing that the corresponding absorption property fails.
\end{proof}

Finally, we note that the sublattice $S=\set{a,b,c,d,y}$  of the slim semimodular lattice (denote it now by $K$) on the left of Figure~\ref{figprt} satisfies $P(8,1)$ and $P(9,2)$ but 
$S$ is not a retract of $K$. Hence, Corollaries \ref{cor:p81} and \ref{cor:p92} only  give necessary conditions but not a characterization of the retracts of slim semimodular lattices.
In fact, we do not know such a characterization.

\end{document}